\documentclass[a4paper, 12pt]{amsart}

\usepackage{amssymb,amsmath}
\usepackage[dvips]{graphics}
\usepackage[dvips]{color}
\usepackage{graphicx}
\usepackage[english]{babel}
\usepackage[latin1]{inputenc}
\usepackage{comment}
\usepackage{todonotes}

\newtheorem{theorem}{Theorem}[section]
\newtheorem{lemma}[theorem]{Lemma}

\newtheorem{corollary}[theorem]{Corollary}

\theoremstyle{definition}
\newtheorem{definition}[theorem]{Definition}
\newtheorem{example}[theorem]{Example}
\newtheorem{remark}[theorem]{Remark}

\numberwithin{equation}{section}

\title[Approximation and quasicontinuity]
{Approximation and quasicontinuity of Besov and Triebel--Lizorkin functions}
\author{Toni Heikkinen, Pekka Koskela and Heli Tuominen}

\DeclareMathOperator*{\essinf}{ess\,inf}

\newcommand\rn{\mathbb R^n}
\newcommand\re{\mathbb R}
\newcommand\rv{\overline{\mathbb R}}
\newcommand\n{\mathbb N}
\newcommand\z{\mathbb Z}
\newcommand\D{\mathbb D}

\newcommand\ph{\varphi}
\newcommand\eps{\varepsilon}
\newcommand\M{\operatorname{\mathcal M}}
\newcommand\cA{\mathcal A}

\newcommand\cD{\mathcal D}
\newcommand\cF{\mathcal F}

\providecommand{\ch}[1]{\text{\raise 2pt \hbox{$\chi$}\kern-0.2pt}_{#1}}

\providecommand{\vint}[1]{\mathchoice
          {\mathop{\vrule width 5pt height 3 pt depth -2.5pt
                  \kern -9.5pt \kern 1pt\intop}\nolimits_{\kern -5pt{#1}}}%
          {\mathop{\vrule width 5pt height 3 pt depth -2.6pt
                  \kern -6pt \intop}\nolimits_{\kern -3pt{#1}}}%
          {\mathop{\vrule width 5pt height 3 pt depth -2.6pt
                  \kern -6pt \intop}\nolimits_{\kern -3pt{#1}}}%
          {\mathop{\vrule width 5pt height 3 pt depth -2.6pt
                  \kern -6pt \intop}\nolimits_{\kern -3pt{#1}}}}
\begin{document}

\begin{abstract}
We show that, for $0<s<1$, $0<p,q<\infty$, Haj\l asz--Besov and Haj\l asz--Triebel--Lizorkin functions can 
be approximated in the norm by discrete median convolutions.
This allows us to show that, for these functions, the limit of medians,
\[
\lim_{r\to 0}m_u^\gamma(B(x,r))=u^*(x),
\] 
exists quasieverywhere and defines a quasicontinuous representative of $u$. 
The above limit exists quasieverywhere also for Haj\l asz functions $u\in M^{s,p}$, $0<s\le 1$, $0<p<\infty$, but approximation of $u$ in $M^{s,p}$ by discrete (median) convolutions is not in general possible.
\end{abstract}



\maketitle

\section{Introduction}
Smooth functions are dense in the classical Sobolev space $W^{1,p}(\rn),$ as 
can be seen by convolution approximation. 
This fact implies that each Sobolev function has a quasicontinuous 
representative and  can be defined everywhere, except for a set of 
$p$-capacity zero, in terms of limits of integral averages. 
Indeed, by the Lebesgue differentiation theorem, almost every point 
$x\in\rn$ is a Lebesgue point,  
\[
\lim_{r\to0}\frac1{|B(x,r)|}\int_{B(x,r)}|u(y)-u(x)|\,dy=0,
\]
which implies that 
\begin{equation}\label{leb point}
\lim _{r\to0}\frac1{|B(x,r)|}\int_{B(x,r)}u(y)\,dy=u(x)
\end{equation}
for such $x$ when $u$ is locally integrable.
For $u\in W^{1,p}(\rn)$, the set of exceptional points where 
\eqref{leb point} fails, is smaller, of $p$-capacity zero. 
The central tools to prove this fact are density of smooth functions and 
capacitary weak-type estimates for the Hardy--Littlewood maximal function, 
see for example \cite{EG}, \cite{Zi}. 
Part of this argument is almost axiomatic and generalizes to other function 
spaces with suitable modifications. 

In this paper, we study approximation of Besov and Triebel--Lizorkin 
functions and the consequent existence of Lebesgue points in a metric space 
$X$ equipped with a doubling measure $\mu$. Because of lack of a 
linear structure, one cannot use the usual convolution approach. 
Hence we employ a discrete convolution, which for a locally integrable function 
$u$ at the scale $r$  is defined as
\[
u_r=\sum_{i}\frac1{|B_i|}\int_{B_i}u\,d\mu\ \ph_i,
\]
where $\{B_i\}_i$ is a covering of $X$ by balls of radius $r$ and 
$\{\ph_i\}_i$ is a partition of unity associated to the covering 
$\{B_i\}_i$, see Section \ref{sec: preliminaries} for details. 
Discrete convolutions have been standard tools in analysis in doubling metric 
spaces starting from the works \cite{CW} and \cite{MS}. 

If our function $u$ fails to be locally integrable, as can happen with 
Besov and Triebel--Lizorkin spaces for small indices $p$ and $q$,
we need to modify the definition of our discrete convolution and use 
medians instead of integral averages.
For $0<\gamma\le 1/2$, the $\gamma$-median of a measurable function
$u\colon X\to\rv$ over a set $A$ of finite measure is
\[
m_u^\gamma(A)=\inf\big\{a\in\re: \mu(\{x\in A: u(x)>a\})< \gamma\mu(A)\big\}.
\]
Median values and median maximal functions have already been studied and used 
in different problems of analysis for quite some time, see   
\cite{FZ}, \cite{Fu}, \cite{GKZ}, \cite{HIT}, \cite{JPW}, \cite{JT}, \cite{J}, \cite{L}, \cite{LP}, \cite{PT}, \cite{St} and \cite{Zh}. 
\medskip

We consider the Haj\l asz--Besov spaces $N^s_{p,q}(X)$ and 
Haj\l asz--Triebel--Lizorkin spa\-ces $M^s_{p,q}(X)$ which were recently 
introduced in \cite{KYZ} and studied for example in \cite{GKZ}, \cite{HeTu}, 
\cite{HeTu2} and \cite{HIT}. 
These spaces consist of those $L^p$-functions $u$ that have a fractional 
$s$-gradient with finite mixed $l^q(L^p(X))$- or $L^p(X,\,l^q)$-norm. 
A sequence $(g_k)_{k\in\z}$ of measurable functions $g_k\colon X\to[0,\infty]$ 
is a fractional $s$-gradient of $u$ if it satisfies the
Haj\l asz-type pointwise inequality
\[
|u(x)-u(y)|\le d(x,y)^s(g_k(x)+g_k(y))
\]
for all $k\in\z$ and almost all $x,y\in X$ satisfying 
$2^{-k-1}\le d(x,y)<2^{-k}$, see Section \ref{sec: preliminaries} for details.
As usual, the homogeneous versions $\dot N^s_{p,q}(X)$ and $\dot M^s_{p,q}(X)$
of our  Haj\l asz--Besov and Haj\l asz--Triebel--Lizorkin spaces are defined
by relaxing the $L^p$-integrability assumption to mere measurability and
finiteness almost everywhere.

In the Euclidean case, $N^s_{p,q}(\rn)=B^s_{p,q}(\rn)$ and $M^s_{p,q}(\rn)=F^s_{p,q}(\rn)$ for all $0<p<\infty$, $0<q\le\infty$, 
$0<s<1$, where $B^s_{p,q}(\rn)$ and $F^s_{p,q}(\rn)$ are the Besov and 
Triebel--Lizorkin spaces defined via an
$L^p$-modulus of smoothness, see \cite{GKZ}. 
Recall also that the Fourier analytic approach gives the same spaces when 
$p>n/(n+s)$ in the Besov case and when $p,q>n/(n+s)$ in the Triebel--Lizorkin 
case. Hence, for such indices, 
our results cover also the case of classical Besov and Triebel--Lizorkin spaces.
\medskip

Our first main result shows that locally Lipschitz functions are dense 
in $N^s_{p,q}(X)$ and in $M^s_{p,q}(X)$. 

\begin{theorem}\label{density of loc lip} 
Let $0<\gamma\le 1/2$, $0<s<1$, $0<p,q<\infty$ and $u\in \dot N^s_{p,q}(X)$. 
Then the discrete $\gamma$-median convolution approximations 
$u_i=u_{2^{-i}}^\gamma$ converge to $u$ in $N^s_{p,q}(X)$ as $i\to\infty$. 
The same result holds with $N^s_{p,q}$ replaced by $M^s_{p,q}$.
\end{theorem}

Concerning earlier approximation results for Besov and Triebel--Lizorkin 
functions, recall first that smooth functions are dense in the classical Besov 
and Triebel--Lizorkin spaces for finite $p,q$ but not necessarily otherwise,
see \cite{Tr} for precise statements.
For metric spaces equipped with doubling measure, the density of Lipschitz 
functions in the homogenous spaces $\dot M^s_{p,q}(X)$ for 
$Q/(Q+s)<p,q<\infty$ has been established in \cite{BSS}. For the so-called
RD-spaces, the density follows already from \cite[Prop 5.21]{HMY} for 
$\dot N^s_{p,q}(X)$ with $Q/(Q+s)<p<\infty$, $0<q<\infty$, and for 
$\dot M^s_{p,q}(X)$ with $Q/(Q+s)<p,q<\infty$. 
Moreover, the density in $N^{Q/p}_{p,p}(X)$, $p>Q$, for Ahlfors $Q$-regular 
metric spaces 
$X$ has been proved in \cite{B}, \cite{Co}.
\medskip

Our second main results says that the limit of medians exists outside a set 
of capacity zero. We say that such points are generalized Lebesgue points of $u$.

\begin{theorem}\label{main thm}
Let $0<s<1$, $0<p,q<\infty$ and $\cF\in\{N^s_{p,q},M^s_{p,q}\}$ or $0<s\le 1$, $0<p<\infty$ and $\cF=M^s_{p,\infty}=M^{s,p}$. 
Let $u\in\dot\cF(X)$. 
Then there exists a set $E\subset X$  with  $C_{\cF}(E)=0$
such that the limit \[\lim_{r\to 0}m_u^{\gamma}(B(x,r))=u^*(x)\] exists for every $x\in X\setminus E$ and $0<\gamma\le 1/2$.
Moreover, $u^*$ is an $\cF$-quasicontinuous representative of $u$.
\end{theorem}

The existence of Lebesgue points for Besov and Triebel--Lizorkin functions in the Euclidean
setting has been studied in \cite{AH}, \cite{D}, \cite{HN}, \cite{Ne90}, 
\cite{N2}, but we are not aware of earlier results in the metric setting. 
\medskip

The paper is organized as follows. We present the notation and definitions 
used in the paper in Section \ref{sec: preliminaries}.
We discuss $\gamma$-medians and their basic properties and define two 
maximal functions related to $\gamma$-medians. The discrete version is a 
modification of the discrete maximal function used for 
locally integrable functions and comparable with the maximal function 
defined via medians over balls. 

We prove Theorem \ref{density of loc lip}, convergence of discrete median 
convolutions in the Haj\-\l asz--Besov and Haj\l asz--Triebel--Lizorkin norms,
 in Sections \ref{sec: appr by med convolutions} and \ref{sec: easy proof}. 
In the former section, the proof is given in the setting of a doubling metric space. 
In the latter section, we give a a substantially
simpler proof under the additional assumption that spheres are 
nonempty. We also give a simple example which shows that discrete (median) 
convolution approximations of a Haj\l asz--Sobolev function $u$ do not necessarily
converge to $u$ in the norm.

In Section \ref{sec: discrete convolutions}, we show that also the discrete 
convolution approximations converge in the Haj\l asz--Besov and 
Haj\l asz--Triebel--Lizorkin norms if $Q/(Q+s)<p<\infty$ 
(where $Q$ is the doubling dimension) in the Besov case and 
$Q/(Q+s)<q<\infty$ in the Triebel--Lizorkin case. 

In Section \ref{sec: capacity}, we discuss capacities connected to the
Haj\l asz--Besov and Haj\l asz--Triebel--Lizorkin spaces. These 
capacities are not necessarily subadditive, but they satisfy the inequality 
$\operatorname {Cap}\big(\cup_{i}E_i\big)^r\le c\sum_{i}\operatorname {Cap}(E_i)^r$
for all sets $E_i\subset X$, $i\in\n$ for $r=\min\{1,q/p\}$ 
(compare this with the Aoki--Rolewicz Theorem for triangle inequality for quasi norms).

Our second main result, Theorem \ref{main thm} is proved in Section 
\ref{sec: qc}. The central tools used in the proof are 
Theorems \ref{median maximal op bounded} and \ref{cap weak type} which show 
that the discrete median maximal operator is bounded on 
Haj\l asz--Besov and Haj\l asz--Triebel--Lizorkin spaces and that a capacitary 
weak-type estimate holds for the median maximal function. 

The final section deals with the existence of 
the classical Lebesgue points defined as limits of integral averages over 
balls. We show that counterparts of results of Section \ref{sec: qc} hold for 
locally integrable functions for suitable parameters of the space. 
In particular, Corollary \ref{leb points for Msp} gives new results for the 
existence Lebesgue points of functions in fractional 
Haj\l asz--Sobolev spaces $M^{s,p}(X)$ for parameters 
$0<s\le1$ and $Q/(Q+s)<p<\infty$.

\section{Preliminaries}\label{sec: preliminaries}

\subsection{Basic assumptions}
In this paper,  $X=(X, d,\mu)$ is a metric measure space equipped with a metric $d$ and a Borel regular,
doubling outer measure $\mu$, for which the measure of every ball is positive and finite.
The doubling property means that there exists a constant $c_d>0$, called {\em the doubling constant}, such that
\[
\mu(B(x,2r))\le c_d\mu(B(x,r))
\]
for every ball $B(x,r)=\{y\in X:d(y,x)<r\}$, where $x\in X$ and $r>0$.

The doubling condition gives an upper bound for the dimension of $X$ because there is a constant $c=c(c_d)$ such that
\[
\frac{\mu(B(y,r))}{\mu(B(x,R))}\ge c\Big(\frac rR\Big)^Q 
\]
for every $0<r\le R$ and $y\in B(x,R)$ with $Q=\log_2 c_d$. 
Below $Q$ refers to this dimension.

{\em The integral average} of a locally integrable function $u$ over a set $A$ of positive and finite measure is
\[
u_A=\vint{A}u\,d\mu=\frac{1}{\mu(A)}\int_A u\,d\mu.
\]

By $\ch{E}$, we denote the characteristic function of a set $E\subset X$ and by $\rv$, the extended real numbers $[-\infty,\infty]$.
$L^0(X)$ is the set of all measurable, almost everywhere finite functions $u\colon X\to\rv$.
In general, $C$ and $c$ are positive constants whose values are not necessarily same at each occurrence.
When we want to stress that the constant depends on other constants or parameters $a,b,\dots$, we write $C=C(a,b,\dots)$.

\subsection{Haj\l asz  spaces, Haj\l asz--Besov and Haj\l asz--Triebel--Lizorkin spaces}
Among several definitions for Besov and Triebel--Lizorkin spaces in metric measure spaces, we use the definitions given by pointwise inequalities in \cite{KYZ}. This definition is motivated by the Haj\l asz--Sobolev spaces $M^{s,p}(X)$, defined for $s=1$, $p\ge1$ in \cite{H} and for fractional scales in \cite{Y}.

\begin{definition}
Let $0<s<\infty$. 
A measurable function $g\ge0$ is an  {\em $s$-gradient} of a function $u\in L^0(X)$ if there exists a set
$E\subset X$ with $\mu(E)=0$ such that for all $x,y\in X\setminus E$,
\begin{equation}\label{eq: gradient}
|u(x)-u(y)|\le d(x,y)^s(g(x)+g(y)).
\end{equation}
The collection of all $s$-gradients of $u$ is denoted by $\cD^s(u)$.

Let $0<p\le\infty$. The {\em homogeneous Haj\l asz space} $\dot{M}^{s,p}(X)$ consists of those measurable functions $u$ for which
\[
\|u\|_{\dot M^{s,p}(X)}=\inf_{g\in\mathcal{D}^s(u)}\|g\|_{L^p(X)}
\]
is finite.
The {\em Haj\l asz space} $M^{s,p}(X)$ is $\dot M^{s,p}(X)\cap L^p(X)$  equipped with the norm
\[
\|u\|_{M^{s,p}(X)}=\|u\|_{L^p(X)}+\|u\|_{\dot M^{s,p}(X)}.
\]
\end{definition}

Recall that for $p>1$, $M^{1,p}(\rn)=W^{1,p}(\rn)$, see \cite{H}, whereas for $n/(n+1)<p\le 1$, 
$M^{1,p}(\rn)$ coincides with the Hardy--Sobolev space $H^{1,p}(\rn)$ by \cite[Thm 1]{KS}.

\begin{definition}
Let $0<s<\infty$.
A sequence of nonnegative measurable functions $(g_k)_{k\in\z}$ is a  {\em fractional $s$-gradient} of a function
$u\in L^0(X)$, if there exists a set $E\subset X$ with $\mu(E)=0$ such that
\begin{equation}\label{frac grad}
|u(x)-u(y)|\le d(x,y)^s(g_k(x)+g_k(y))
\end{equation}
for all $k\in\z$ and all $x,y\in X\setminus E$ satisfying $2^{-k-1}\le d(x,y)<2^{-k}$.
The collection of all fractional $s$-gradients of $u$ is denoted by $\D^s(u)$.
\end{definition}

The next two lemmas for fractional $s$-gradients follow easily from the definition. 
The corresponding results for $1$-gradients have been proved in \cite[Lemmas 2.4]{KiMa} and \cite[Lemma 2.6]{KL}. 
We leave the proofs for the reader.

\begin{lemma}\label{max lemma}
Let $u,v\in L^0(X)$, $(g_k)_{k\in Z}\in\D^s(u)$ and $(h_k)_{k\in Z}\in\D^s(v)$. 
Then the sequence $(\max\{g_k,h_k\})_{k\in \z}$ is a fractional $s$-gradient of 
the functions $\max\{u,v\}$ and $\min\{u,v\}$.
\end{lemma}


\begin{lemma}\label{sup lemma} 
Let $u_i\in L^0(X)$ and $(g_{i,k})_{k\in\z}\in \D^s(u_i)$, $i\in\n$. 
Let $u=\sup_{i\in \n} u_i$ and $(g_k)_{k\in\z}=(\sup_{i\in\n}g_{i,k})_{k\in\z}$.
If $u\in L^0(X)$, then $(g_k)_{k\in\z}\in \D^s(u)$.
\end{lemma}

For $0<p,q\le\infty$ and a sequence $(f_k)_{k\in\z}$ of measurable
functions, we define
\[
\big\|(f_k)_{k\in\z}\big\|_{L^p(X,\,l^q)}
=\big\|\|(f_k)_{k\in\z}\|_{l^q}\big\|_{L^p(X)}
\]
and
\[
\big\|(f_k)_{k\in\z}\big\|_{l^q(L^p(X))}
=\big\|\big(\|f_k\|_{L^p(X)}\big)_{k\in\z}\big\|_{l^q},
\]
where
\[
\big\|(f_k)_{k\in\z}\big\|_{l^{q}}
=
\begin{cases}
\big(\sum_{k\in\z}|f_{k}|^{q}\big)^{1/q},&\quad\text{if }0<q<\infty, \\
\;\sup_{k\in\z}|f_{k}|,&\quad\text{if }q=\infty.
\end{cases}
\]

\begin{definition}
Let $0<s<\infty$ and $0<p,q\le\infty$.
The  {\em homogeneous Haj\l asz--Triebel--Lizorkin space} $\dot M_{p,q}^s(X)$ consists of those functions
$u\in L^0(X)$, for which the (semi)norm
\[
\|u\|_{\dot M_{p,q}^s(X)}
=\inf_{(g_k)\in\D^s(u)}\|(g_k)\|_{L^p(X,\,l^q)}
\]
is finite. 
The  {\em Haj\l asz--Triebel--Lizorkin space} $M_{p,q}^s(X)$ is $\dot M_{p,q}^s(X)\cap L^p(X)$
equipped with the norm
\[
\|u\|_{M_{p,q}^s(X)}=\|u\|_{L^p(X)}+\|u\|_{\dot M_{p,q}^s(X)}.
\]
Notice that $M^s_{p,\infty}(X)=M^{s,p}(X)$, see \cite[Prop.\ 2.1]{KYZ} for a simple proof of this fact.

Similarly, the  {\em homogeneous Haj\l asz--Besov space} $\dot N_{p,q}^s(X)$ consists of those functions
$u\in L^0(X)$, for which
\[
\|u\|_{\dot N_{p,q}^s(X)}=\inf_{(g_k)\in\D^s(u)}\|(g_k)\|_{l^q(L^p(X))}
\]
is finite, and the {\em Haj\l asz--Besov space} $N_{p,q}^s(X)$ is $\dot N_{p,q}^s(X)\cap L^p(X)$
equipped with the norm
\[
\|u\|_{N_{p,q}^s(X)}=\|u\|_{L^p(X)}+\|u\|_{\dot N_{p,q}^s(X)}.
\]
\end{definition}
For $0<s<1$, $0<p,q\le \infty$, the spaces $N^s_{p,q}(\rn)$ and $M^s_{p,q}(\rn)$ coincide with the classical Besov and Triebel--Lizorkin spaces defined via differences ($L^p$-modulus of smoothness),  see \cite{GKZ}.

When $0<p<1$ or $0<q<1$, the (semi)norms defined above are actually quasi-(semi)norms, but for simplicity we call them, as well as other quasi-seminorms in this paper, just norms.

\subsection{Inequalities} 
We will frequently use the elementary inequality
\begin{equation}\label{elem ie}
\sum_{i\in\z} a_i\le \Big(\sum_{i\in\z} a_i^{\beta}\Big)^{1/\beta},
\end{equation}
which holds whenever $a_i\ge 0$ for all $i$ and $0<\beta\le 1$. 
\smallskip

By \eqref{elem ie}, if $0<p<1$ and by the Minkowski inequality, if $p\ge 1$, we have
\begin{equation}\label{norm ie}
\big\|\sum_{i\in\z} f_i\big\|_{L^p(X)}^{\tilde p}\le \sum_{i\in\z} \|f_i\|_{L^p(X)}^{\tilde p},
\end{equation}
where $\tilde p=\min\{1,p\}$.

The H\"older inequality and \eqref{elem ie} easily imply
the following lemma, which is used to estimate the norms of fractional gradients.

\begin{lemma}[\cite{HIT}, Lemma 3.1]\label{summing lemma}
Let $1<a<\infty$, $0<b<\infty$ and $c_k\ge 0$, $k\in\z$. There exists a constant $C=C(a,b)$ such that
\[
\sum_{k\in\z}\Big(\sum_{j\in\z}a^{-|j-k|}c_j\Big)^b\le C\sum_{j\in\z}c_j^b.
\]
\end{lemma}

The Fefferman--Stein vector-valued maximal theorem (see \cite{FS}, \cite{GLY}, \cite{Sa}) states that, for $1<p,q\le\infty$,
there exists a constant $C=C(c_d,p,q)$ such that
\begin{equation}\label{Fefferman-Stein}
\|(\M u_k)\|_{L^p(X;l^q)}\le C\|(u_k)\|_{L^p(X;l^q)}
\end{equation}
for every $(u_k)_{k\in\z}\in L^p(X;l^q)$. Here $\M$ is the usual Hardy--Littlewood maximal operator.

\subsection{$\gamma$-median} 
Recall from the introduction that for $0<\gamma\le 1/2$, the $\gamma$-median of a measurable function
$u\colon X\to\rv$ over a set $A$ of finite measure is
\[
m_u^\gamma(A)=\inf\big\{a\in\re: \mu(\{x\in A: u(x)>a\})< \gamma\mu(A)\big\},
\]
and note that if $u\in L^0(A)$ and $0<\mu(A)<\infty$, then $m_u^\gamma(A)$ is finite.

In the following lemma, we list some basic properties of the $\gamma$-median. 
We leave the quite straightforward proof for the reader, who can also look at \cite{PT} where most of the properties are proved in the Euclidean setting. Properties (a), (b), (d), (f), (g) and (h) follow from \cite[Propositions 1.1 and 1.2]{PT} and (i) and (j) from \cite[Theorem 2.1]{PT}. The remaining properties (c) and (e) follow immediately from the definition.

\begin{lemma}\label{median lemma} 
The $\gamma$-median has the following properties:
\begin{itemize}
\item[(a)] If $\gamma\le\gamma'$, then $m_{u}^{\gamma}(A)\ge m_{u}^{\gamma'}(A)$.
\item[(b)] If $u\le v$ almost everywhere, then  $m_{u}^{\gamma}(A)\le m_{v}^{\gamma}(A)$.
\item[(c)] If $A\subset B$ and $\mu(B)\le C\mu(A)$, then $m_{u}^{\gamma}(A)\le m_{u}^{\gamma/C}(B)$.
\item[(d)] If $c\in\mathbb{R}$, then $m_u^\gamma(A)+c=m_{u+c}^\gamma(A)$.
\item[(e)] If $c>0$, then $m_{c\,u}^\gamma(A)=c\,m_{u}^\gamma(A)$.
\item[(f)] $|m_{u}^\gamma(A)|\le m_{|u|}^\gamma(A)$. 
\item[(g)] $m_{u+v}^\gamma(A)\le m_{u}^{\gamma/2}(A)+m_{v}^{\gamma/2}(A)$.
\item[(h)] For every $p>0$, 
\[
m_{|u|}^\gamma(A)\le \Big(\gamma^{-1}\vint{A}|u|^p\,d\mu\Big)^{1/p}.
\]
\item[(i)] If $u$ is continuous, then for every $x\in X$,
\[
\lim_{r\to 0} m_{u}^\gamma(B(x,r))=u(x).
\]
\item[(j)]  If $u\in L^0(X)$, then 
there exists a set $E$ with $\mu(E)=0$ such that \[\lim_{r\to 0} m_{u}^\gamma(B(x,r))=u(x)\]
 for every $0<\gamma\le 1/2$ and $x\in X\setminus E$. 
\end{itemize}
\end{lemma}

Property (j) above says that medians over small balls behave like 
integral averages of locally integrable functions at Lebesgue points. 

\begin{definition}
Let $u\in L^0(A)$. 
A point $x$ is a \emph{generalized Lebesgue point} of $u$, if 
\[
\lim_{r\to 0} m_{u}^\gamma(B(x,r))=u(x)
\]
for all $0<\gamma\le1/2$.
\end{definition}

\subsection{Maximal functions}
We use two maximal operators defined using medians. The first one is defined 
in the usual way by taking a supremum of medians over balls and the 
second one is a discrete version defined using discrete convolutions.

\begin{definition}
Let $0<\gamma\le1/2$. 
The \emph{$\gamma$-median maximal function} of a function $u\in L^0(X)$ is $\M^\gamma u\colon X\to\rv$,
\[
\M^\gamma u(x)=\sup_{r>0}\,m_{|u|}^\gamma(B(x,r)).
\]
\end{definition}
Note that by Lemma \ref{median lemma} (h) and (j), for any $p>0$,
\begin{equation}\label{M iso}
u(x)\le \M^\gamma u(x)\le \big(\gamma^{-1}\M u^p(x)\big)^{1/p} 
\end{equation}
for almost all $x\in X$. 
It follows from the Hardy--Littlewood maximal theorem that, for any $p>0$, there exists a constant $C=C(\gamma,p,c_d)$ such that
\begin{equation}\label{median maximal operator bounded in Lp}
\|\M^\gamma u\|_{L^p(X)}\le C\|u\|_{L^p(X)},
\end{equation}
for every $u\in L^p(X)$. 
More generally, \eqref{M iso} and the Fefferman--Stein maximal theorem, \eqref{Fefferman-Stein}, imply that, for all  $0<p,q\le\infty$, there exists a constant $C=C(\gamma,p,q,c_d)$ such that
\begin{equation}\label{Lplq boundedness of median maximal op}
\|(\M^\gamma  u_k)\|_{L^p(X;l^q)}\le C\|(u_k)\|_{L^p(X;l^q)}
\end{equation}
for every $(u_k)_{k\in\z}\in L^p(X;l^q)$.
\medskip

For the discrete maximal function, we first recall a definition of a discrete convolution. 
Discrete convolutions are basic tools in harmonic analysis in homogeneous spaces, see for example \cite{CW} and \cite{MS}.
The discrete maximal function, which can be seen as a smooth version of the Hardy--Littlewood maximal function, was introduced in \cite{KL}. 

Let $r>0$ and let $\{B_i\}_{i\in I}$, where $I\subset\n$, be a covering of $X$ by balls of radius $r$ such that
\[
\sum_{i\in I}\ch{2B_i}\le C(c_d).
\]
For such a covering, there exist $C/r$-Lipschitz functions $\ph_i$, $i\in I$, and a constant $C=C(c_d)$ such that $0\le\ph_i\le 1$, 
$\ph_i=0$ outside $2B_i$ and $\ph_i\ge C^{-1}$ on $B_i$ for all $i$ and $\sum_{i\in I}\ph_i=1$. 
Such a collection of functions is called a \emph{partition of unity} subordinate to the covering $\{B_i\}_{i\in I}$.
A \emph{discrete convolution} of a locally integrable function $u$ at the scale $r$  is
\begin{equation}\label{discrete convolution}
u_r=\sum_{i\in I}u_{B_i}\ph_i
\end{equation}
and the \emph{discrete  maximal function} of $u$ is $\M^* u\colon X\to\rv$,
\[
\M^*u(x)=\sup_{k\in\z}|u|_{2^k}(x).
\]

Similarly, we define median versions of a discrete convolution and a discrete maximal function. 
Below, balls $B_i$ and functions $\ph_i$ are as above.
\begin{definition}
Let $0<\gamma\le1/2$. 
A \emph{discrete $\gamma$-median convolution} of a function $u\in L^0(X)$ at scale $r>0$ is
\[
u_r^\gamma=\sum_{i\in I}m_{u}^\gamma(B_i)\ph_i,
\]
and 
the \emph{discrete $\gamma$-median maximal function} of $u$ is $\M^{\gamma,*} u\colon X\to\rv$,
\[
\M^{\gamma,*}u(x)=\sup_{k\in\z}|u|_{2^k}^\gamma(x).
\]
\end{definition}
Notice that taking the usual discrete convolution is a linear operation in the sense that
\[
(u+v)_r=u_r+v_r,
\]
but the discrete median convolutions only satisfy
\[
(u+v)_r^\gamma\le u_r^{\gamma/2}+v_r^{\gamma/2}.
\]

By  \cite[Lemma 3.1]{KL}, there exists a constant $C=C(c_d)$ such that
\[
C^{-1}\M u\le \M^*u\le C\M u.
\]
Similarly, it is easy to see that there is a constant $C>0$ such that
\begin{equation}\label{comparability for median maximal functions}
\M^{\gamma}u\le C\M^{\gamma/C,*}u\le C^2\M^{\gamma/C^2}u.
\end{equation}

As suprema of continuous functions, the discrete maximal functions are lower semicontinuous and hence measurable.

\begin{remark}\label{finiteness remark} 
If $u\in L^0(X)$, then $\M^\gamma u$ may be infinite at every point. 
However, if $\M^\gamma u(x_0)<\infty$ at some point $x_0\in X$, then $\M^\gamma u<\infty$ almost everywhere. 
Indeed, at almost every point, $\lim_{r\to 0}m_{|u|}^\gamma(B(x,r))=|u(x)|<\infty,$ which implies that there exists $r_0>0$ such that $m_{|u|}^\gamma(B(x,r))\le |u(x)|+1$ for $r<r_0$. On the other hand, if $r\ge r_0$, then $B(x,r)\subset B(x_0,r+d(x,x_0))$ and there exists a constant $C=C(c_d,r_0,d(x,x_0))$ such that $\mu(B(x_0,r+d(x,x_0)))\le C\mu(B(x,r))$, which implies that
$m_{|u|}^\gamma(B(x,r))\le\M^{\gamma/C}u(x_0).$
\end{remark}

\section{Approximation by discrete median convolutions}\label{sec: appr by med convolutions}

In this section, we prove Theorem \ref{density of loc lip} in the setting of a doubling metric space without any additional assumption on the measure. For the proof, we need a couple of lemmas including a Leibniz type rule for fractional $s$-gradients, a Sobolev-Poincar\'e type inequality and a pointwise estimate for $\gamma$-medians. 

\begin{lemma}\label{lemma 1} 
Let $0<s\le1$ and let $S\subset X$ be a measurable set.
Let $u\colon X\to\re$ be a measurable function with $(g_k)_{k\in\z}\in \D^s(u)$ and
let $\ph$ be a bounded $L$-Lipschitz function supported in $S$.
Let $i\in\z$. Then the sequence $(h_k)_{k\in\z}$, where
\begin{equation}\label{fracgrad}
h_{k}=
\begin{cases}
\big(\|\varphi\|_{\infty} g_k+L2^{k(s-1)}|u|\big)\ch{S}, &\text{ when } k>i\\
  2^{(k+1)s}\|\varphi\|_{\infty}|u|\ch{S}, &\text{ when } k\le i
\end{cases},
\end{equation}
is a fractional $s$-gradient of function $u\varphi$.
\end{lemma}

\begin{proof}
By \cite[Lemma 3.10]{HIT}, the sequence $(\hat h_k)_{k\in\z}$, where
\[
\hat h_k=\big(\|\varphi\|_{\infty} g_k+L2^{k(s-1)}|u|\big)\ch{S}
\]
is a fractional $s$-gradient of $u\varphi$.
On the other hand, by denoting 
\[
\tilde h_k=2^{(k+1)s}\|\varphi\|_{\infty}|u|\ch{S},
\]
we have
\[
\begin{split}
|(u\varphi)(x)-(u\varphi)(y)|
&\le \|\varphi\|_{\infty}|u(x)|\ch{S}(x)+\|\varphi\|_{\infty}|u(y)|\ch{S}(y)\\
&\le d(x,y)^s\big(\tilde h_k(x)+\tilde h_k(y)\big),
\end{split}
\]
when $2^{-k-1}\le d(x,y)<2^{-k}$, which implies that $(\tilde h_k)_{k\in\z} \in\D^s(u\varphi)$.
Thus, for every $i\in\z$, $(h_k)_{k\in\z} \in\D^s(u\varphi)$.
\end{proof}

Notice that $\tilde h_k< \hat h_k$, when $2^{-k}>2\|\varphi\|_\infty/L$. 
Hence we usually choose $i$ in \eqref{fracgrad}  such that $2^{-i}\approx\|\varphi\|_\infty/L$.
\medskip

The definition of fractional $s$-gradients implies the validity of various Poincar\'e type and pointwise inequalities. 
Our next lemma follows from the Sobolev--Poincar\'e inequality formulated in terms of integral averages in 
\cite[Lemma 2.1]{GKZ} and Lemma \ref{median lemma} (h).

\begin{lemma}\label{med poincare}
Let $0<\gamma\le 1/2$ and $0<s,t<\infty$. 
There exist constants $0<s'<s$ and $C>0$ such that
\begin{equation}\label{eq: med poincare}
\inf_{c\in \re}\ m^\gamma_{|u-c|}(B(x,2^{-k}))
\le C2^{-ks}\sum_{l\ge k-2}2^{(k-l)s'}
\Big(\,\vint{B(x,2^{-k+1})}g_l^{t}\,d\mu\Big)^{1/t}
\end{equation}
for all $u\in L^0(X)$, $(g_k)_k\in\D^s(u)$, $x\in X$ and $k\in\z$. 
\end{lemma}

Notice that by Lemma \ref{median lemma} (b), (d) and (f),
\begin{equation}\label{med quasimin}
m^\gamma_{|u-m^\gamma_u(A)|}(A) \le 2\inf_{c\in \re}\ m^\gamma_{|u-c|}(A)
\end{equation}
whenever $0<\mu(A)<\infty$ and $u\in L^0(A)$.
\medskip

Using a chaining argument and \eqref{eq: med poincare}, we obtain the following pointwise estimate.

\begin{lemma}\label{med pointwise} 
Let $0<\gamma\le 1/2$ and $0<s, t<\infty$. There exist constants $0<s'<s$ and $C>0$ such that
\[
|u(x)-m^\gamma_u(B(y,2^{-i}))|\le C2^{-is}\sum_{l>i-5}2^{(i-l)s'}(\M g_l^t(x))^{1/t}
\]
for all $u\in L^0(X)$, $(g_k)_k\in \D^s(u)$,
$y\in X$, $i\in\z$ and almost every $x\in B(y,2^{-i+1})$.
\end{lemma}

\begin{proof}
Let $x\in B(y,2^{-i+1})$ be a generalized Lebesgue point of $u$. 
Then, by Lemma \ref{median lemma} (j), (d), (f), (c), \eqref{med quasimin} and \eqref{eq: med poincare}, we obtain
\[
\begin{split}
|u(x)-m^\gamma_u(B(x,2^{-i+2}))|&\le\sum_{k>i-3}|m^\gamma_u(B(x,2^{-k-1}))-m^\gamma_u(B(x,2^{-k}))|\\
&\le \sum_{k>i-3}\,m^\gamma_{|u-m^\gamma_u(B(x,2^{-k}))|}(B(x,2^{-k-1}))\\
&\le C\sum_{k>i-3}\,m^{\gamma/C}_{|u-m^\gamma_u(B(x,2^{-k}))|}(B(x,2^{-k}))\\
&\le C\sum_{k>i-3}2^{-ks}\sum_{l\ge k-2}2^{(k-l)s'}
\Big(\,\vint{B(x,2^{-k+1})}g_l^{t}\,d\mu\Big)^{1/t}\\
&\le C\sum_{l>i-5}2^{-ls'}(\M g_l^t(x))^{1/t} \sum_{i-3<k\le l+2}2^{k(s'-s)}\\
&\le C2^{-is}\sum_{l>i-5}2^{(i-l)s'}(\M g_l^t(x))^{1/t}.
\end{split}
\]
Since $B(y,2^{-i})\subset B(x,2^{-i+2})$,  Lemma \ref{median lemma} (d), (f), (c), \eqref{med quasimin} and \eqref{eq: med poincare} imply that
\[
\begin{split}
|m^\gamma_u(B(y,2^{-i}))-m^\gamma_u(B(x,2^{-i+2}))|
&\le Cm^{\gamma/C}_{|u-u_{B(x,2^{-i+2})}|}(B(x,2^{-i+2}))\\
&\le C2^{-is}\sum_{l> i-5}2^{(i-l)s'}
\Big(\,\vint{B(x,2^{-i+3})}g_l^{t}\,d\mu\Big)^{1/t}\\
&\le C2^{-is}\sum_{l> i-5}2^{(i-l)s'}
(\M g_l^t(x))^{1/t}.
\end{split}
\]
The claim follows by combining the above estimates.
\end{proof}

\begin{lemma}\label{lemma 3} 
Let $0<p,q,\delta<\infty$ and $i\in\z$.
\begin{itemize}
\item[a)] If $0<t<p$, then
\[
\Big\|\sum_{l>i}2^{-(l-i)\delta}(\M g_l^t)^{1/t}\Big\|_{L^p(X)}
\le C\Big(\sum_{l>i}\|g_l\|_{L^p(X)}^q\Big)^{1/q}.
\]
\item[b)] If $0<t<\min\{p,q\}$, then
\[
\Big\|\sum_{l>i}2^{-(l-i)\delta}(\M g_l^t)^{1/t}\Big\|_{L^p(X)}
\le C\Big\|\Big(\sum_{l>i}g_l^q\Big)^{1/q}\Big\|_{L^p(X)}.
\]
\end{itemize}
\end{lemma}

\begin{proof} 
a) Let $r=\min\{1,p\}$. By the H\"older inequality if $p>1$ and by the elementary inequality \eqref{elem ie} if $p\le 1$, we have
\[
\bigg(\sum_{l>i}2^{-(l-i)\delta}(\M g_l^t)^{1/t}\bigg)^p
\le C\sum_{l>i}2^{-(l-i)\delta r}(\M g_l^t)^{p/t}.
\]
Hence
\[
\begin{split}
\Big\|\sum_{l>i}2^{-(l-i)\delta}(\M g_l^t)^{1/t}\Big\|_{L^p(X)}^p
&\le \sum_{l>i}2^{-(l-i)\delta r} \big\|(\M g_l^t)\big\|_{L^{p/t}(X)}^{p/t}\\
&\le C\sum_{l>i}2^{-(l-i)\delta r} \|g_l\|_{L^{p}(X)}^{p}\\
&\le C\Big(\sum_{l>i}\|g_l\|_{L^p(X)}^q\Big)^{p/q},
\end{split}
\]
where the second estimate follows from the Hardy--Littlewood maximal theorem and the last one from the H\"older inequality if $q>p$ and from \eqref{elem ie} if $q\le p$.
\medskip

b) By the H\"older inequality if $q>1$ and by \eqref{elem ie} if $q\le 1$,
\[
\sum_{l>i}2^{-(l-i)\delta}(\M g_l^t)^{1/t}
\le C\Big(\sum_{l>i}(\M g_l^t)^{q/t}\Big)^{1/q}.
\]
Hence, by the Fefferman--Stein maximal theorem,
\[
\begin{split}
\Big\|\sum_{l>i}2^{-(l-i)\delta}\big(\M g_l^t\big)^{1/t}\Big\|_{L^p(X)}
&\le C\Big\|\Big(\sum_{l>i}\big(\M g_l^t\big)^{q/t}\Big)^{t/q}\Big\|_{L^{p/t}(X)}^{1/t}\\
&\le C\Big\|\Big(\sum_{l>i} g_l^q\Big)^{t/q}\Big\|_{L^{p/t}(X)}^{1/t}\\
&= C\Big\|\Big(\sum_{l>i} g_l^q\Big)^{1/q}\Big\|_{L^{p}(X)}.\qedhere
\end{split}
\]
\end{proof}

\begin{proof}[Proof of Theorem \ref{density of loc lip}]
Let $u\in \dot N^s_{p,q}(X)$ and let $u_i=u_{2^{-i}}^\gamma$, $i\in\n$. 
By the definition of the discrete $\gamma$-median convolution and the properties of the functions $\ph_j$, 
\[
u-u_i=\sum_{j\in J} \big(u-m_u^\gamma(B_j)\big)\ph_j.
\]
Since each $\ph_j$ is $C2^i$-Lipschitz and supported in $2B_j$, Lemma \ref{lemma 1} implies
that the sequence $(Ch^j_k)_{k\in\z}$, where
\[
h^j_{k}=
\begin{cases}
(g_k+2^{i+k(s-1)}|u-m_u^\gamma(B_j)|)\ch{2 B_j}, &\text{ when } k>i\\
2^{ks}|u-m_u^\gamma(B_j)|\ch{2 B_j}, &\text{ when } k\le i
\end{cases},
\]
is a fractional $s$-gradient of the function $(u-m^\gamma_u(B_j))\varphi_j$. 

Let $(g_k)_{k\in\z}\in \D^s(u)$ and $0<t<\min\{p,q\}$.
Then, by Lemma \ref{lemma 2} and by
the bounded overlap of the balls $2B_j$, the sequence $(Ch_k)_{k\in\z}$, where
\[
h_{k}=
\begin{cases}
g_k+2^{(k-i)(s-1)}\sum_{l>i-5}2^{(i-l)s'}(\M g_l^t)^{1/t}, &\text{ when } k>i\\
2^{(k-i)s}\sum_{l>i-5}2^{(i-l)s'}(\M g_l^t)^{1/t}, &\text{ when } k\le i
\end{cases},
\]
is a fractional $s$-gradient of $u-u_i$.  
 
By Lemma \ref{lemma 3}, we have 
\[
\Big(\sum_{k\in\z}\|h_k\|_{L^p(X)}^q\Big)^{1/q}
\le C\Big(\sum_{l>i-5}\|g_l\|_{L^p(X)}^q\Big)^{1/q}
\]
and by Lemmas \ref{med pointwise} and \ref{lemma 3},
\[
\|u-u_{i}\|_{L^p(X)}
\le C2^{-is}\Big(\sum_{l>i-5}\|g_l\|_{L^p(X)}^q\Big)^{1/q}.
\]
Thus, if $u\in \dot N^s_{p,q}(X)$, then $\| u-u_{i}\|_{N^s_{p,q}(X)}\to 0$ as $i\to\infty$. 
\medskip

Similarly, by Lemma \ref{lemma 3}, 
\[
\Big\|\Big(\sum_{k\in\z}h_k^q\Big)^{1/q}\Big\|_{L^p(X)}
\le C\Big\|\Big(\sum_{l>i-5}g_l^q\Big)^{1/q}\Big\|_{L^p(X)}
\]
and by Lemmas \ref{med pointwise} and \ref{lemma 3},
\[
\|u-u_{i}\|_{L^p(X)}\le C2^{-is}
\Big\|\Big(\sum_{l>i-5}g_l^q\Big)^{1/q}\Big\|_{L^p(X)}.
\]
So, if $u\in \dot M^{s}_{p,q}(X)$, then $\| u-u_{i}\|_{M^s_{p,q}(X)}\to 0$ as $i\to\infty$. 
\end{proof}

The following example shows that discrete (median) convolution approximations of a Haj\l asz function $u\in M^{1,p}(X)$ do not necessarily converge to $u$ in $M^{1,p}(X)$.
\begin{example}
Let $X=[0,1]$ and let $d$ and $\mu$ be the restrictions of the Euclidean distance and the Lebesgue measure to $[0,1]$. 
Let $i\in\n$. For $j=0,1,\dots,2^{i}$, let $B_j=B(j2^{-i},2^{-i})$ and 
\[
\varphi_j(x)=\max\{0,1-2^{-i-1}d(x,B(j2^{-i},2^{-i-2}))\}.
\] 
Then $\{\ph_j\}_j$ is a partition of unity subordinate to the covering $\{B_j\}_j$ of $X$. 
If $x\in B(j2^{-i},2^{-i-2})$, then $\varphi_j(x)=1$ and $\varphi_k(x)=0$ for $k\neq j$.

Let $u\colon X\to\re$, $u(x)=x$. 
Then, for each $x\in B(j2^{-i},2^{-i-2})$, 
\[
u^\gamma_{2^{-i}}(x)=\sum_k m^\gamma_u(B_k)\varphi_k(x)=m^\gamma_u(B_j).
\] 
Hence, if $x,y\in B(j2^{-i},2^{-i-2})$, then
\[
(u-u^\gamma_{2^{-i}})(x)-(u-u^\gamma_{2^{-i}})(y)=u(x)-u(y)=x-y.
\]
It follows that any $g\in \cD^1(u-u^\gamma_{2^{-i}})$ must satisfy $g(x)\ge1/2$ for almost every $x\in \cup_j B(j2^{-i},2^{-i-2})$. 
Thus $\|u-u^\gamma_{2^{-i}}\|_{M^{1,p}(X)}\not\to 0$ as $i\to\infty$.

The same argument shows that the usual discrete convolutions $u_{2^{-i}}$ do not converge to $u$ in $M^{1,p}(X)$ either. 
\end{example}

\section{Another proof of Theorem \ref{density of loc lip}}\label{sec: easy proof}
In this section, we give a simpler proof for Theorem \ref{density of loc lip} under the assumption that
the underlying space $X$ has the nonempty spheres property. 
This proof is completely elementary and avoids the use of a chaining argument, a Sobolev--Poincar\'e inequality and the Fefferman--Stein maximal theorem.

\begin{definition}\label{nonempty spheres}
A metric space $X$ has the \emph{nonempty spheres property}, if there exists $R>0$ such that, for every $x\in X$ and every $0<r<R$, the set $\{y\in X: d(x,y)=r\}$ is nonempty. 
\end{definition}
Note that the nonempty spheres property implies that annuli have positive measure: Let $x\in X$, $0<r<R$, $0<\eps<r$ and let 
$A=B(x,r)\setminus B(x,r-\eps)$. By the assumption, there is $y$ such that $d(x,y)=r-\eps/2$. 
Now $B_y=B(y,\eps/2)\subset A$ which shows that $\mu(A)\ge\mu(B_y)>0$.
 
The nonempty spheres property allows us to prove the following estimate, which is an improved version
of Lemma \ref{med pointwise}.
Notice that, in contrast to the proof of Lemma \ref{med pointwise}, neither a chaining argument nor a Sobolev--Poincar\'e
inequality is needed.

\begin{lemma}\label{lemma 2 med 2} 
Assume that $X$ has the nonempty spheres property.
Let $0<\gamma\le 1/2$ and $0<s<\infty$. 
Let $u\in L^0(X)$ and let $(g_k)_{k\in\z}\in \D^s(u)$ be such that $g_k\in L^0(X)$ for every $k\in\z$.
Then there exists a constant $C=C(c_d,s)$ such that inequality
\[
|u(x)-m_u^\gamma(B(y,2^{-i}))|\le C2^{-is}\M^{\gamma/C}\tilde g_i(x),
\]
where $\tilde g_i=\max_{i-4\le k\le i} g_k$, holds for every $y\in X$, $i\in\z$ with $2^{-i+3}<R$ and almost every $x\in B(y,2^{-i+1})$.
\end{lemma}

\begin{proof} 
Denote by $E$ the set outside of which \eqref{frac grad} holds.
Let $x\in B(y,2^{-i+1})\setminus E$ be a generalized Lebesgue point of $g_k$ for every $k$. 
By Lemma \ref{median lemma} (j), almost every point is such a point.
Let 
\[
A=B(y,2^{-i+3})\setminus B(y,2^{-i+2}).
\] 
Then $A\subset B(x,2^{-i+4})$ and the nonempty spheres property and the
doubling condition imply that $\mu(B(x,2^{-i+4}))\le C\mu(A)$.
By the triangle inequality, we have
\begin{equation}\label{eq}
|u(x)-m_u^\gamma(B(y,2^{-i}))|\le |u(x)-m_u^\gamma(A)|+|m_u^\gamma(A)-m_u^\gamma(B(y,2^{-i}))|.
\end{equation}
We begin by estimating the first term on the right-hand side of \eqref{eq}. 
If $z\in A\setminus E$, then $2^{-i}\le d(x,z)<2^{-i+4}$, and hence 
\[
|u(x)-u(z)|\le 2^{(-i+4)s}(\tilde g_i(x)+\tilde g_i(z)).
\] 
Hence, using Lemma \ref{median lemma} (b)-(f), we have
\[
\begin{split}
|u(x)-m_u^\gamma(A)|&=|m_{u-u(x)}^\gamma(A)|
\le m_{|u-u(x)|}^\gamma(A)\\
&\le C2^{-is}\big(m_{\tilde g_i}^\gamma(A)+\tilde g_i(x)\big)\\
&\le C2^{-is}\big(m_{\tilde g_i}^{\gamma/C}(B(x,2^{-i+4}))+\tilde g_i(x)\big)\\
&\le C2^{-is}\M^{\gamma/C}\tilde g_i(x).
\end{split}
\]

Next, we estimate the second term on the right-hand side of \eqref{eq}. By the same argument as above, 
\[
\begin{split}
|u(z)-m_u^\gamma(A)| \le C2^{-is}\big(m_{\tilde g_i}^\gamma(A)+\tilde g_i(z)\big)
\end{split}
\]
for every $z\in B(y,2^{-i})\setminus E$. Hence we obtain
\[
\begin{split}
|m_u^\gamma(A)-m_u^\gamma(B(y,2^{-i}))| &\le m_{|u-m_u^\gamma(A)|}^\gamma(B(y,2^{-i}))\\ 
&\le C2^{-is}\big(m_{\tilde g_i}^{\gamma}(A)+m_{\tilde g_i}^\gamma(B(y,2^{-i}))\big)\\
&\le C2^{-is}m_{\tilde g_i}^{\gamma/C}(B(x,2^{-i+4}))\\
&\le C2^{-is} \M^{\gamma/C}\tilde g_i (x),
\end{split}
\]
and the claim follows.
\end{proof}

\begin{proof}[Proof of Theorem \ref{density of loc lip} (when $X$ satisfies the nonempty spheres property)]

Let $i\in\z$ be such that $2^{-i+3}<R$. Let $u\in L^0(X)$ and $(g_k)_{k\in\z}\in \D^s(u)$ with $g_k\in L^0(X)$ for every $k\in\z$.
It follows from Lemmas \ref{lemma 1} and \ref{lemma 2 med 2} and from the bounded overlap of the balls $2B_j$ that the sequence 
$(Ch_k)_{k\in\z}$, where
\[
h_{k}=
\begin{cases}
g_k+2^{(k-i)(s-1)}\M^{\gamma/C} \tilde g_i, &\text{ when } k>i\\
2^{(k-i)s}\M^{\gamma/C}\tilde g_i, &\text{ when } k\le i
\end{cases},
\]
and $\tilde g_i=\max_{i-4\le k\le i} g_k$, is a fractional $s$-gradient of $u-u_i$.  
By the boundedness of $\M^{\gamma/C}$ in 
$L^p$, \eqref{median maximal operator bounded in Lp}, we have
\[
\begin{split}
\|u-u_i\|_{\dot N^s_{p,q}(X)}&\le C
\Big(\sum_{k\in\z}\|h_k\|_{L^p(X)}^q\Big)^{1/q}\\
&\le C\Big(\sum_{k>i}\|g_k\|_{L^p(X)}^q\Big)^{1/q}+C\|\M^{\gamma/C}\tilde g_i\|_{L^p(X)}\\
&\le C\Big(\sum_{k>i-5}\|g_k\|_{L^p(X)}^q\Big)^{1/q}
\end{split}
\]
and
\[
\begin{split}
\|u-u_i\|_{\dot M^s_{p,q}(X)}
&\le C\Big\|\Big(\sum_{k\in\z}h_k^q\Big)^{1/q}\Big\|_{L^p(X)}\\
&\le C\Big\|\Big(\sum_{k>i}g_k^q\Big)^{1/q}\Big\|_{L^p(X)}+C\big\|\M^{\gamma/C}\tilde g_i\big\|_{L^p(X)}\\
&\le C\Big\|\Big(\sum_{k>i-5}g_k^q\Big)^{1/q}\Big\|_{L^p(X)}.
\end{split}
\]
Moreover, by Lemma \ref{lemma 2 med 2} and \eqref{median maximal operator bounded in Lp}, 
\[
\|u-u_i\|_{L^p(X)}\le C2^{-is}\|\M^{\gamma/C}\tilde g_i\|_{L^p(X)}\le C2^{-is}\|\tilde g_i\|_{L^p(X)}.
\]
The claim follows from these estimates.
\end{proof}

\section{Approximation by discrete convolutions}\label{sec: discrete convolutions}
The main result of this section is Theorem \ref{density of loc lip 2}, which is a counterpart of Theorem \ref{density of loc lip} for usual discrete convolutions. It shows that discrete convolutions converge to the locally integrable function in Haj\l asz--Besov and Haj\l asz--Triebel--Lizorkin norm if $Q/(Q+s)<p<\infty$ in the Besov case and $Q/(Q+s)<q<\infty$ in the Triebel--Lizorkin case. 
Recall from \eqref{discrete convolution} that the discrete convolution of a locally integrable function $u$ at the scale $r$  is 
$u_r=\sum_{i\in I}u_{B_i}\ph_i$.

\begin{theorem}\label{density of loc lip 2} 
Let $0<s<1$ and $Q/(Q+s)<p<\infty$.

\begin{itemize}
\item[a)] If $0<q<\infty$, then 
\[
\lim_{i\to\infty}\|u_{2^{-i}}-u\|_{N^s_{p,q}(X)}=0
\] 
for every $u\in \dot N^s_{p,q}(X)$.
\item[b)] If $Q/(Q+s)<q<\infty$, then 
\[
\lim_{i\to\infty}\|u_{2^{-i}}-u\|_{M^s_{p,q}(X)}=0
\] 
for every $u\in \dot M^s_{p,q}(X)$.
\end{itemize}
\end{theorem}
The proof is almost the same as the proof of Theorem \ref{density of loc lip} given in Section \ref{sec: appr by med convolutions}. Instead of Lemma \ref{med poincare} we use Lemma \ref{poincare} which follows from \cite[Lemma 2.1]{GKZ}.

\begin{lemma}\label{poincare}
Let $0<s<\infty$ and $t>Q/(Q+s)$. Then there exist constants $0<s'<s$ and $C>0$ such that
\begin{equation}\label{Poincare for averages}
\begin{split}
\vint{B(x,2^{-k})}|u-u_{B(x,2^{-k})}|\,d\mu
\le C2^{-ks}\sum_{l\ge k-2}2^{(k-l)s'}
\Big(\,\vint{B(x,2^{-k+1})}g_l^{t}\,d\mu\Big)^{1/t}
\end{split}
\end{equation}
for all locally integrable functions $u$, all $(g_k)_{k\in\z}\in\D^s(u)$, $x\in X$ and $k\in\z$. 
In particular, if $p>Q/(Q+s)$, then \eqref{Poincare for averages} holds for every $u\in N^s_{p,q}(X)\cup M^s_{p,q}(X)$.
\end{lemma}

A standard chaining argument and \eqref{Poincare for averages} imply the following pointwise estimate.

\begin{lemma}\label{lemma 2} 
Let $0<s<\infty$ and $Q/(Q+s)<t<\infty$. Then there exist constants $C>0$ and $0<s'<s$ such that
\[
|u(x)-u_{B(y,2^{-i})}|\le C2^{-is}\sum_{l>i-5}2^{(i-l)s'}\big(\M g_l^t(x)\big)^{1/t}
\]
for all locally integrable functions $u$, all $(g_k)_{k\in\z}\in \D^s(u)$,
$y\in X$, $i\in\z$ and almost every $x\in B(y,2^{-i+1})$.
\end{lemma}

\begin{proof}[Proof of Theorem \ref{density of loc lip 2}]
a) Let $u\in\dot N^s_{p,q}(X)$, $(g_k)_{k\in\z}\in\D^s(u)\cap l^q(L^p(X))$ and let $Q/(Q+s)<t<p$.
Since
\[
u-u_{2^{-i}}=\sum_{j\in J}(u-u_{B_j})\varphi_j,
\]
it follows from Lemmas \ref{lemma 1}, \ref{lemma 2} and from the bounded overlap of the balls $2B_j$ that the sequence
$(Ch_k)_{k\in\z}$, where
\[
h_{k}=
\begin{cases}
g_k+2^{(k-i)(s-1)}\sum_{l>i-5}2^{(i-l)s'}(\M g_l^t)^{1/t}, &\text{ when } k>i\\
2^{(k-i)s}\sum_{l>i-5}2^{(i-l)s'}(\M g_l^t)^{1/t}, &\text{ when } k\le i
\end{cases},
\]
is a fractional $s$-gradient of our function $u-u_{2^{-i}}$.  
By Lemma \ref{lemma 3}, we have 
\[
\bigg(\sum_{k\in\z}\|h_k\|_{L^p(X)}^q\bigg)^{1/q}
\le C\bigg(\sum_{l>i-5}\|g_l\|_{L^p(X)}^q\bigg)^{1/q}
\]
and by Lemmas \ref{lemma 2} and \ref{lemma 3},
\[
\|u-u_{2^{-i}}\|_{L^p(X)}
\le C2^{-is}\bigg(\sum_{l>i-5}\|g_l\|_{L^p(X)}^q\bigg)^{1/q}.
\]
Thus, $\| u-u_{2^{-i}}\|_{N^s_{p,q}(X)}\to 0$ as $i\to\infty$. 
\medskip

b)  Let $u\in\dot M^s_{p,q}(X)$, $(g_k)_{k\in\z}\in\D^s(u)\cap L^p(X,l^q)$ and $Q/(Q+s)<t<\min\{p,q\}$.
Then the sequence $(Ch_k)_{k\in\z}$, where $h_k$ is as above, is a fractional $s$-gradient of $u-u_{2^{-i}}$.  
By Lemma \ref{lemma 3}, we have
\[
\Big\|\Big(\sum_{k\in\z}h_k^q\Big)^{1/q}\Big\|_{L^p(X)}
\le C\Big\|\Big(\sum_{l>i-5}g_l^q\Big)^{1/q}\Big\|_{L^p(X)}
\]
and by Lemmas \ref{lemma 2} and \ref{lemma 3},
\[
\|u-u_{2^{-i}}\|_{L^p(X)}\le C2^{-is}
\Big\|\Big(\sum_{l>i-5}g_l^q\Big)^{1/q}\Big\|_{L^p(X)}.
\]
Hence $\| u-u_{2^{-i}}\|_{M^s_{p,q}(X)}\to 0$ as $i\to\infty$. 
\end{proof}

\section{Capacity}\label{sec: capacity}
In this section, we define Haj\l asz--Besov and Haj\l asz--Triebel--Lizorkin capacities and prove some of their basic properties.
In the Euclidean setting, Besov and Triebel--Lizorkin capacities are studied and used for example in \cite{A1}, \cite{A2}, \cite{AH}, \cite{AHS}, \cite{AX}, \cite{D}, \cite{HN}, \cite{Ne90}, \cite{Ne92}, \cite{Ne96}, \cite{Sto} and in metric spaces in \cite{B}, \cite{Co} (with less than three indices in the space).  

\begin{definition}
Let $0<s<\infty$, $0<p,q\le \infty$ and $\cF\in\{N^s_{p,q}, M^s_{p,q}\}$.
The $\cF$-capacity of a set $E\subset X$ is
\[
C_{\cF}(E)=\inf\Big\{\|u\|_{\cF}^p: u\in\mathcal A_{\cF}(E)\Big\},
\]
where 
\[
\mathcal A_{\cF}(E)=\big\{u\in \cF: u\ge 1 \text{ on a neighbourhood of } E\big\}
\] 
is a set of admissible functions for the capacity.
We say that a property holds \emph{$\cF$-quasieverywhere} if it holds outside a set of $\cF$-capacity zero.
\end{definition}

\begin{remark} 
Lemma \ref{max lemma} easily implies that
\[
C_{\cF}(E)=\inf\Big\{\|u\|_{\cF}^p: u\in\mathcal A_{\cF}'(E)\Big\},
\]
where
$\mathcal A_{\cF}'(E)=\{u\in \cA_\cF(E): 0\le u\le 1\}$. 
\end{remark}

\begin{remark}\label{cap remark 2} 
It is easy to see that the $\cF$-capacity is an outer capacity, which means that
\[
C_{\cF}(E)=\inf\big\{C_{\cF}(U): U\supset E,\ U \text{ is open}\big\}.
\]
\end{remark}

The $\cF$-capacity is not generally subadditive, but for most purposes, it suffices that it
satisfies \eqref{eq: r-subadd} for some $r>0$. Even in the classical case, countable subadditivity for Besov-capacity is known only when $p\le q$, see \cite{A1}.

\begin{lemma}\label{r-subadd} 
Let $0<s<\infty$, $0<p\le \infty$, $0<p\le \infty$  and 
let $\cF\in\{N^s_{p,q}, M^s_{p,q}\}$.
Then there are constants $c\ge 1$ and $0<r\le 1$ such that
\begin{equation}\label{eq: r-subadd}
C_{\cF}\big(\bigcup_{i\in\n}E_i\big)^r\le c\sum_{i\in\n} C_{\cF}(E_i)^r
\end{equation}
for all sets $E_i\subset X$, $i\in\n$.
Actually, \eqref{eq: r-subadd} holds with $r=\min\{1,q/p\}$.
\end{lemma}

\begin{proof} 
Let $r=\min\{1,q/p\}$ and $E=\cup_{i\in\n} E_i$. Let $\eps>0$. 

We prove the case $\cF=N^s_{p,q}$ first. 
We may assume that $ \sum_{i\in\n} C_{N^s_{p,q}}(E_i)^{r}<\infty$.
There are functions $u_i\in \cA'_{N^s_{p,q}}(E_i)$ with $(g_{i,k})_{k\in\z}\in \D^s(u_i)$ such that 
\[
\big(\|u_i\|_{L^p(X)}+\|(g_{i,k})_k\|_{l^q(L^p(X))}\big)^{pr}< C_{N^s_{p,q}}(E_i)^r+2^{-i}\eps.
\] 
Then $u=\sup_{i\in\n} u_i\in \cA'_{N^s_{p,q}}(E)$ and the sequence 
$(g_k)_{k\in\z}=(\sup_{i\in\n}g_{i,k})_{k\in\z}$ is a fractional $s$-gradient of $u$. 
We have $\|g_k\|_{L^p(X)}^p\le \sum_{i\in\n}\|g_{i,k}\|_{L^p(X)}^p$, for every $k$, and so
\[
\begin{split}
\|(g_k)_k\|_{l^q(L^p(X))}^{pr}&= \Big\|\big(\|g_{k}\|_{L^p(X)}^p\big)_k\Big\|_{l^{q/p}}^{r}
\le \Big\|\sum_{i\in\n}\big(\|g_{i,k}\|_{L^p(X)}^p\big)_k\Big\|_{l^{q/p}}^{r}\\
&\le\sum_{i\in\n} \Big\|\big(\|g_{i,k}\|_{L^p(X)}^p\big)_k\Big\|_{l^{q/p}}^{r}
=\sum_{i\in\n}\|(g_{i,k})_k\|_{l^q(L^p(X))}^{pr}
\end{split}
\]
Here we also used the fact that
\[
\big\|\sum_{k\in\n}(a^k_i)_{i\in\z}\big \|_{l^{q/p}}^r
\le \sum_{k\in\n}\big\|(a^k_i)_{i\in\z}\big\|_{l^{q/p}}^r
\] 
for all $(a^k_i)_{i\in\z}\in l^{q/p}$, $k\in\n$. 
Since $\|u\|_{L^p(X)}^{pr}\le \sum_{i\in\n}\|u_i\|_{L^p(X)}^{pr}$, we have that
\[
\begin{split}
C_{N^s_{p,q}}(E)^{r} 
&\le 2^{pr}\big(\|u\|_{L^p(X)}^{pr}+\|(g_{k})_k\|_{l^q(L^p(X))}^{pr}\big)\\
&\le2^{pr}\sum_{i\in\n}\big(\|u_i\|_{L^p(X)}^{pr}+ \|(g_{i,k})_k\|_{l^q(L^p(X))}^{pr}\big)\\
&\le 2^{pr+1}\sum_{i\in\n}\big(\|u_i\|_{L^p(X)}+ \|(g_{i,k})_k\|_{l^q(L^p(X))}\big)^{pr}\\
&\le 2^{pr+1}\Big(\sum_{i\in\n}C_{N^s_{p,q}}(E_i)^r+\eps\Big).
\end{split}
\]
The claim follows by letting $\eps\to 0$. 
\medskip

Assume then that $ \sum_{i\in\n} C_{M^s_{p,q}}(E_i)^{r}<\infty$.
Let $u_i\in \cA'_{M^s_{p,q}}(E_i)$ and $(g_{i,k})_{k\in\z}\in \D^s(u_i)$ be such that 
\[
\big(\|u_i\|_{L^p(X)}+\|(g_{i,k})_k\|_{L^p(X,l^q)}\big)^{pr}< C_{M^s_{p,q}}(E_i)^r+2^{-i}\eps. 
\]
Then $u=\sup_{i\in\n} u_i$ belongs to $\cA'_{M^s_{p,q}}(E)$ and $(g_k)_{k\in\z}=(\sup_{i\in\n}g_{i,k})_{k\in\z}$ is a fractional $s$-gradient of $u$. 
Since  $g_k^q\le\sum_i g_{i,k}^q$, for every $k$, it follows that
\[
\begin{split}
\|(g_k)_k\|_{L^p(X,l^q)}^{pr}
&= \Big\|\big(\sum_{k\in\z} g_k^q\big)^{1/q}\Big\|_{L^p(X)}^{pr}
\le \Big\|\sum_{k\in\z} \sum_{i\in\n} g_{i,k}^q\Big\|_{L^{p/q}(X)}^{pr/q}\\
&= \Big\|\sum_{i\in\n} \sum_{k\in\z}g_{i,k}^q\Big\|_{L^{p/q}(X)}^{\min\{1,p/q\}}
\le\sum_{i\in\n}  \Big\|\sum_{k\in\z}g_{i,k}^q\Big\|_{L^{p/q}(X)}^{\min\{1,p/q\}}\\
&=\sum_{i\in\n}  \Big\|\big(\sum_{k\in\z}g_{i,k}^q\big)^{1/q}\Big\|_{L^{p}(X)}^{pr}
=\sum_{i\in\n} \|(g_{i,k})_k\|_{L^p(X,l^q)}^{pr}.
\end{split}
\]
The rest of the proof is the same as in the Besov case.
\end{proof}



\section{Quasicontinuity and generalized Lebesgue points}\label{sec: qc}
In this section, we prove the second main theorem of the paper, 
Theorem \ref{main thm}, which shows that quasievery point is a generalized 
Lebesgue points of any given Haj\l asz--Besov or Haj\l asz--Triebel--Lizorkin function and that the limit of medians gives a quasicontinuous representative of the function. 

As usual, we say that a function $u$ is $\cF$-\emph{quasicontinuous}, if for every $\eps>0$, there exists a set $U$ such that $C_{\cF}(U)<\eps$ and the restriction of $u$ to $X\setminus U$ is continuous. By Remark \ref{cap remark 2}, the set $U$ can be chosen to be open. 

The main ingredient of the proof of Theorem \ref{main thm} is a capacitary 
weak-type estimate for the median maximal function proved in Theorem \ref{cap weak type} below. This estimate follows immediately from the fact that the discrete median maximal operator is bounded on $\cF(X)$ (Theorem \ref{median maximal op bounded}). 
We begin with a lemma, which gives a formula for a fractional $s$-gradient of the discrete median maximal function.

\begin{lemma}\label{frac grad for med max}
Let $0<\gamma\le 1/2$, $0<s<1$ and $0<t<\infty$.
Let $u\in L^0(X)$ and $(g_k)_{k\in\z}\in\D^s(u)$. 
Then there exist constants $C>0$ and $0<s'<s$ such that $(C\tilde g_k)_{k\in\z}$, where 
\[
\tilde g_k=\sum_{l\in\z}2^{-|l-k|\delta}(\M g_l^t)^{1/t}
\]
and $\delta=\min\{s',1-s\}$, is a fractional $s$-gradient of the discrete $\gamma$-median convolution approximation 
$u^\gamma_{2^{-i}}$ for every $i\in\z$. 
Consequently, if $\M^{\gamma,*} u\not\equiv\infty$, then $(C\tilde g_k)_{k\in\z}$ is a fractional $s$-gradient of $\M^{\gamma,*} u$.
\end{lemma}

\begin{proof}
By the proof of Theorem \ref{density of loc lip}, the sequence $(Ch_k)_{k\in\z}$, where
\[
h_{k}=
\begin{cases}
g_k+2^{(i-k)(1-s)}\sum_{l>i-5}2^{(i-l)s'}(\M g_l^t)^{1/t}, &\text{ when } k>i\\
2^{(k-i)s}\sum_{l>i-5}2^{(i-l)s'}(\M g_l^t)^{1/t}, &\text{ when } k\le i
\end{cases},
\]
is a fractional $s$-gradient of our function $u-u_{2^{-i}}^\gamma$.  It follows that $(C\tilde h_k)_{k\in\z}$, where
\[
\tilde h_{k}=
\begin{cases}
g_k+2^{(i-k)(1-s)}\sum_{l>i-5}2^{(i-l)s'}(\M g_l^t)^{1/t}, &\text{ when } k>i\\
g_k+2^{(k-i)s}\sum_{l>i-5}2^{(i-l)s'}(\M g_l^t)^{1/t}, &\text{ when } k\le i
\end{cases},
\]
is a fractional $s$-gradient of $u_{2^{-i}}^\gamma$. Clearly, $g_k\le \tilde g_k$ almost everywhere for every $k$.

If $k>i$, then
\[
\begin{split}
2^{(i-k)(1-s)}\sum_{i-5<l<k}2^{(i-l)s'}(\M g_l^t)^{1/t}
&\le C\sum_{i-5<l<k}2^{(l-k)(1-s)}(\M g_l^t)^{1/t}
\end{split}
\]
and
\[
\begin{split}
2^{(i-k)(1-s)}\sum_{l\ge k}2^{(i-l)s'}(\M g_l^t)^{1/t}
&\le \sum_{l\ge k}2^{(k-l)s'}(\M g_l^t)^{1/t}.
\end{split}
\]
If $k\le i$, then
\[
\begin{split}
2^{(k-i)s}\sum_{l>i-5 }2^{(i-l)s'}(\M g_l^t)^{1/t}
&= 2^{(k-i)(s-s')}\sum_{l>i-5}2^{
(k-l)s'}(\M g_l^t)^{1/t}\\
&\le \sum_{l>k-5}2^{
(k-l)s'}(\M g_l^t)^{1/t}.
\end{split}
\]
Thus, $\tilde h_k\le C\tilde g_k$ almost everywhere, for every $k\in\z$, and so $(\tilde g_k)_{k\in\z}\in \D^s(u_{2^{-i}}^\gamma)$.

If $\M^{\gamma,*} u\not\equiv\infty$, then by Remark \ref{finiteness remark}, $\M^{\gamma,*} u\in L^0(X)$.  
Hence, by Lemma \ref{sup lemma}, $(C\tilde g_k)_{k\in\z}\in\D^s(\M^{\gamma,*}u)$.
\end{proof}

From the proof of the above lemma, we see that if $g$ is an $s$-gradient of $u$ and $0<t<\infty$, then there exists a constant 
$C>0$ such that $C(\M g^t)^{1/t}$ is an $s$-gradient of $u_{2^{-i}}^\gamma$ for every $i$.
Below, we show that an even better result holds. We need the following lemma, which is a special case of Lemma
\ref{lemma 1}.

\begin{lemma}\label{gradient for product} 
Let $0<s\le 1$ and let $S\subset X$ be a measurable set.
Let $u\colon X\to\re$ be a measurable function with $g\in \cD^s(u)$ and
let $\ph$ be a bounded $L$-Lipschitz function supported in $S$.
Then 
\[
h=\big(\|\varphi\|_{\infty} g+(2\|\varphi\|_{\infty}\big)^{1-s}L^s |u|)\ch S \in \cD^s(u\varphi).
\]
\end{lemma}

\begin{lemma}\label{gradient for discrete median maximal function} 
Let $0<s\le 1$  and $0<\gamma\le 1/2$.
Let $u\in L^0(X)$ and $g\in\cD^s(u)\cap L^0(X)$. 
Then there exists a constant $C>0$ such that $C\M^{\gamma/C}g$ is an $s$-gradient of $u^\gamma_r$
for every $r>0$. 
Consequently, if $\M^{\gamma,*} u\not\equiv\infty$, then $C\M^{\gamma/C} g$ is an $s$-gradient of $\M^{\gamma,*} u$.
\end{lemma}

\begin{proof} 
By the definition of the discrete $\gamma$-median convolution and the properties of the functions $\ph_i$, 
\[
u^\gamma_r=u+\sum_{i\in\n}(m^\gamma_u(B_i)-u)\ph_i,
\]
and by Lemma \ref{gradient for product}, the function
\[
(g+Cr^{-s}|u-m^\gamma_u(B_i)|)\ch{2B_i}
\]
is an $s$-gradient of our function $(m^\gamma_u(B_i)-u)\varphi_i$ for each $i$.

Let $x\in 2B_i\setminus E$, where $E$ is the exceptional set for \eqref{eq: gradient}. 
Using Lemma \ref{median lemma}, we obtain
\[
\begin{split}
|u(x)-m_u^{\gamma}(B_i)|&\le m_{|u-u(x)|}^{\gamma}(B_i)
\le Cr^s\big( m_{g}^{\gamma/C}(B_i)+g(x)\big)\\
&\le Cr^s\big( m_{g}^{\gamma/C}(B(x,3r))+g(x)\big)\\
&\le Cr^s(g(x)+\M^{\gamma/C} g(x)).
\end{split}
\]
Since $g(x)\le\M^{\gamma/C}g(x)$ for almost every $x$ and since the balls $2B_i$ have bounded overlap, it follows that
\[
C\M^{\gamma/C} g\in \cD^s(u^\gamma_r),
\]
for every $r>0$. 

Assume then that $\M^{\gamma,*} u\not\equiv\infty$. 
Then, by Remark \ref{finiteness remark}, $\M^{\gamma,*}u\in L^0(X)$, and so, by Lemma \ref{sup lemma}, 
\[
C\M^{\gamma/C} g\in \cD^s(\M^{\gamma,*} u),
\]
from which the second claim follows.
\end{proof}

\begin{remark} 
If $X$ has the nonempty spheres property (Definition \ref{nonempty spheres}) with $R=\infty$, 
then the proof of Theorem \ref{density of loc lip} given in Section \ref{sec: easy proof} shows that, for $0<\gamma\le 1/2$ and 
$0<s\le 1,$ there exists a constant $C>0$ such that $(C\tilde g_k)_{k\in\z}$, where 
\[
\tilde g_k=\sup_{l\in\z} 2^{-|l-k|\delta}\M^{\gamma/C}g_l
\]
and $\delta=\min\{s,1-s\}$, is a fractional $s$-gradient of $\M^{\gamma,*}u$, whenever 
$u\in L^0(X)$, $(g_k)_{k\in\z}\in\D^s(u)$ and $\M^{\gamma,*}u\not\equiv\infty$.
\end{remark}

The proof of the following lemma is essentially contained in the proofs of Theorems 4.5 and 4.6 in \cite{HeTu}, but
for the convenience of the reader, we repeat the proof.

\begin{lemma}\label{hk lemma}
Let $0<p,q,\delta<\infty$ and $0<t<\min\{p,q\}$. Let $(g_k)_{k\in\z}\in l^q(L^p(X))\cup L^p(X,l^q)$ and let
\[
h_k=\sum_{l\in\z}2^{-|l-k|\delta}(\M g_l^t)^{1/t}
\]
for every $k\in\z$. 
Then there is a constant $C=C(p,q,t,\delta,c_d)$ such that
\[
\|(h_k)_{k\in\z}\|_{l^q(L^p(X))}\le C\|(g_k)_{k\in\z}\|_{l^q(L^p(X))},
\]
if $(g_k)_{k\in\z}\in l^q(L^p(X))$, and
\[
\|(h_k)_{k\in\z}\|_{L^p(X,l^q)}\le C\|(g_k)_{k\in\z}\|_{L^p(X,l^q)},
\]
if $(g_k)_{k\in\z}\in L^p(X,l^q)$.
\end{lemma}

\begin{proof}
Let $\tilde p=\min\{1,p\}$. Using \eqref{norm ie} and the Hardy--Littlewood maximal theorem, we have 
\[
\|h_k\|_{L^p(X)}^{\tilde p}\le\sum_{l\in\z}2^{-|l-k|\delta\tilde p}\|(\M g_l^t)^{1/t}\|_{L^p(X)}^{\tilde p}
\le C\sum_{l\in\z}2^{-|l-k|\delta\tilde p}\|g_l\|_{L^p(X)}^{\tilde p}.
\]
Hence, by Lemma \ref{summing lemma},
\[
\|(h_k)\|_{l^q(L^p(X))}^{q}\le C\sum_{k\in\z}\Big(
\sum_{l\in\z}2^{-|l-k|\delta\tilde p}\|g_l\|_{L^p(X)}^{\tilde p}\Big)^{q/\tilde p}
\le C
\sum_{l\in\z}\|g_l\|_{L^p(X)}^{q}
\]
and the first claim follows.

Since 
\[
\|(h_k)\|_{l^q}^q=\sum_{k\in\z}\Big(
\sum_{l\in\z}2^{-|l-k|\delta}(\M g_l^t)^{1/t}\Big)^{q}\le C\sum_{k\in\z}\big(\M g_k^t\big)^{q/t}
\]
by Lemma \ref{summing lemma}, the Fefferman--Stein maximal theorem implies that
\[
\|(h_k)\|_{L^p(X;l^q)}\le C \|(M g_k^t)\|_{L^{p/t}(X;\, l^{q/t})}^{1/t}
\le C\|(g_k^t)_{k\in\z}\|_{L^{p/t}(X;\, l^{q/t})}^{1/t},
=C\|(g_k)\|_{L^{p}(X,l^{q})},
\]
which proves the second claim.
\end{proof}

Next we show that the discrete median maximal operator is bounded on 
$\cF(X)$ and use this result to prove a capacitary weak-type estimate for the 
median maximal function.

\begin{theorem}\label{median maximal op bounded}
Let $0<\gamma\le 1/2$. Let $0<s<1$, $0<p<\infty$, $0<q<\infty$ and 
$\cF\in\{N^s_{p,q},M^s_{p,q}\}$ or $0<s\le 1$, $0<p<\infty$ and $\cF=M^s_{p,\infty}=M^{s,p}$. 
Then there exists a constant $C=C(c_d,s,p,q,\gamma)$ such that
\[
\|\M^{\gamma,*}u\|_{\cF(X)}\le C\|u\|_{\cF(X)}
\]
for all $u\in \cF(X)$. 
\end{theorem}

\begin{proof}
By the boundedness of the discrete median maximal operator on $L^p$, which follows from \eqref{comparability for median maximal functions} and \eqref{median maximal operator bounded in Lp}, we have 
\[
\|\M^{\gamma,*}u\|_{L^p(X)}\le C\|u\|_{L^p(X)}.
\]
To estimate the $s$-gradient, suppose first that $0<s<1$, $0<p,q<\infty$, $u\in N^s_{p,q}(X)$ and $(g_k)_{k\in\z}\in \D^s(u)$. 
Let $0<t<\min\{p,q\}$. By Lemma \ref{frac grad for med max}, the sequence $(C\tilde g_k)_{k\in\z}$, where
\[
\tilde g_k=\sum_{l\in\z}2^{-|l-k|\delta}(\M g_l^t)^{1/t}
\]
and $\delta=\min\{s',1-s\}$,
is a fractional $s$-gradient of $\M^{\gamma,*}u$ and by Lemma \ref{hk lemma},
\[
\|(\tilde g_k)_{k\in\z}\|_{l^q(L^p(X))}\le C\|(g_k)_{k\in\z}\|_{l^q(L^p(X))}.
\]
Thus, $\|\M^{\gamma,*}u\|_{\dot N^s_{p,q}(X)}\le C\|u\|_{\dot N^s_{p,q}(X)}$. 

The proof for the estimate $\|\M^{\gamma,*}u\|_{\dot M^s_{p,q}(X)}\le C\|u\|_{\dot M^s_{p,q}(X)}$ is analogous.
Finally, the estimate $\|\M^{\gamma,*}u\|_{\dot M^{s,p}(X)}\le C\|u\|_{\dot M^{s,p}(X)}$ follows from 
Lemma \ref{gradient for discrete median maximal function} and \eqref{median maximal operator bounded in Lp}.
\end{proof}

\begin{theorem}\label{cap weak type}
Let $0<\gamma\le 1/2$. 
Let $0<s<1$, $0<p<\infty$, $0<q<\infty$ and $\cF\in\{N^s_{p,q},M^s_{p,q}\}$ or $0<s\le 1$, $0<p<\infty$ and 
$\cF=M^s_{p,\infty}=M^{s,p}$. 
Then there exists a constant $C=C(c_d,s,p,q,\gamma)$ such that
\[
C_{\cF}\big(\{x\in X: \M^\gamma u(x) >\lambda\}\big)
\le C\lambda^{-p}\|u\|_{\cF(X)}^p
\]
for all $u\in \cF(X)$ and for all $\lambda>0$.
\end{theorem}

\begin{proof}
By \eqref{comparability for median maximal functions}, $\M^\gamma u\le C\M^{\gamma/C,*}u$. 
Hence, Theorem \ref{median maximal op bounded} together with the lower semicontinuity of $\M^{\gamma/C,*} u$ implies that
\[
\begin{split}
C_{\cF}\big(\{x\in X: \M^\gamma u(x) >\lambda\}\big)
&\le C_{\cF}\big(\{x\in X: C\lambda^{-1}\M^{\gamma/C,*} u(x) >1 \}\big)\\
&\le C\lambda^{-p}\|\M^{\gamma/C,*}u\|_{\cF(X)}^p\\
&\le C\lambda^{-p}\|u\|_{\cF(X)}^p,
\end{split}
\]
and the claim follows.
\end{proof}

\begin{proof}[Proof of Theorem \ref{main thm}]
We assume first that $u\in\cF(X)$. The general case, $u\in\dot\cF(X)$, then follows by a localization argument.

We have to show that the limit $\lim_{r\to 0}m_u^{\gamma}(B(x,r))=u^*(x)$ exists outside a set of zero 
$\cF$-capacity and that $u^*$ is an $\cF$-quasicontinuous representative of $u$.

By Theorem \ref{density of loc lip}, if $0<q<\infty$, and by \cite[Proposition 4.5]{SYY}, if $q=\infty$, 
there are continuous functions $u_i$, $i\in\n$, such that 
\[
\|u-u_i\|_{\cF(X)}^p<2^{-(1+p)i}
\] 
for each $i\in\n$. 

For each $n,i,k\in\n$, let 
\[
A_{n,i}=\big\{x: \M^{1/(2n)}(u-u_i)(x)>2^{-i}\big\}
\quad\text{ and }\quad
B_{n,k}=\bigcup_{i\ge k} A_{n,i}.
\]
We will show that the set
\[
E=\bigcup_{n\in\n} E_n,
\] 
where
\[
E_n=\bigcap_{k\in\n} B_{n,k}
\] 
is of zero capacity and that the limit of medians of $u$ exists outside $E$. 

By Theorem \ref{cap weak type}, the capacity of each set $A_{n,i}$ has an upper bound,
\[
C_{\cF}(A_{n,i})\le C2^{ip}\|u-u_i\|_{\cF(X)}^p<C2^{-i}
\]
and hence, by \eqref{eq: r-subadd}, 
\[
C_{\cF}(B_{n,k})\le C2^{-k}.
\]
This implies that $C_{\cF}(E)=0$.
\medskip
 
Next we show that the continuous approximations $u_i$ of $u$ converge uniformly to a continuous function outside $B_{n,k}$ 
and that the limit function is the limit of $\gamma$-medians of $u$. 
Let $1/n\le \gamma\le 1/2$. By Lemma \ref{median lemma}, we have
\[
\begin{split}
|u_i(x)-m_u^\gamma(B(x,r))|
&=|m_{(u-u_i(x))}^\gamma(B(x,r))|\le m_{|u-u_i(x)|}^{1/n}(B(x,r))\\
&\le m_{|u-u_i|}^{1/(2n)}(B(x,r))+m_{|u_i-u_i(x)|}^{1/(2n)}(B(x,r)),
\end{split}
\]
which, together with the fact that $m_{|u_i-u_i(x)|}^{1/(2n)}(B(x,r))\to 0$ as $r\to 0$ by the continuity of $u_i$, implies that for each 
$x\in X\setminus A_{n,i}$,
\[
\limsup_{r\to 0}|u_i(x)-m_u^\gamma(B(x,r))|\le \M^{1/(2n)}(u-u_i)(x)\le 2^{-i}.
\]
Now, for each $x\in X\setminus B_{n,k}$ and $i,j\ge k$, we have
\[
\begin{split}
|u_i(x)-u_j(x)|
&\le \limsup_{r\to 0}\,\Big(|u_i(x)-m_u^\gamma(B(x,r))|+|u_j(x)-m_u^\gamma(B(x,r))|\Big)\\
&\le 2^{-i}+2^{-j},
\end{split}
\]
which implies that $(u_i)$ converges uniformly in $X\setminus B_{n,k}$ to a continuous function $v$.
Since
\[
\limsup_{r\to 0}|v(x)-m_u^\gamma(B(x,r))|
\le |v(x)-u_i(x)|+\limsup_{r\to 0}|u_i(x)-m_u^\gamma(B(x,r))|,
\]
it follows that
\[
v(x)=\lim_{r\to0}m_u^\gamma(B(x,r))=u^*(x)
\]
for every $x\in X\setminus B_{n,k}$ and $1/n\le \gamma\le 1/2$. Hence the limit
\[
\lim_{r\to0}m_u^\gamma(B(x,r))=u^*(x)
\] 
exist for every $x\in X\setminus E$ and $0<\gamma\le 1/2$. 
\medskip

Now, assume that $u\in\dot\cF(X)$. Let $x_0\in X$ and, for each $k\in\n$, let $\ph_k$ be a Lipschitz function such that $\ph=1$ on $B(x_0,k)$ and $\ph=0$ in $X\setminus B(x_0,2k)$. Then, by \cite[Lemma 3.10 and Remark 3.11]{HIT}, $u_k=u\ph_k\in\cF(X)$. 
By the first part of the proof, for every $k$, there exists a set $E_k$ with $C_{\cF}(E_k)=0$ such that the limit 
\[
\lim_{r\to0}m_{u_k}^\gamma(B(x,r))=u_k^*(x)
\] 
exists for every $x\in X\setminus E_k$ and $0<\gamma\le 1/2$.  Since 
\[
\lim_{r\to0}m_{u}^\gamma(B(x,r))=\lim_{r\to0}m_{u_k}^\gamma(B(x,r))
\] 
in $B(x_0,k)$, it follows that the limit
\[
\lim_{r\to0}m_{u}^\gamma(B(x,r))=u^*(x)
\] 
exists for every $x\in X\setminus \cup_{k\in\n} E_k$. By \eqref{eq: r-subadd}, $C_{\cF}(\cup_{k\in\n}E_k)=0$. 

Let $\eps>0$. By the first part of the proof, for every $k$, there exists a set $U_k$ with $C_{\cF}(U_k)<2^{-k}\eps$ such that 
$u_k^*$ is continuous in $X\setminus U_k$. 
Since $u^*=u_k^*$ in $B(x_0,k)$, it follows that $u^*$ is continuous in
$X\setminus \cup_{k\in\n}U_k$. By \eqref{eq: r-subadd}, $C_{\cF}(\cup_{k\in\n}U_k)<C\eps$. 
Thus, $u^*$ is $\cF$-quasicontinuous.
\end{proof}

\section{Lebesgue points}\label{sec: integral averages}

For restricted values of $p$ and $q$, counterparts of Theorems \ref{main thm}, 
\ref{median maximal op bounded} and \ref{cap weak type} hold also for integral averages. 
The main result of this section is the following.

\begin{theorem}\label{leb points}
Let $\cF=N^s_{p,q}$, where $0<s<1$, $Q/(Q+s)<p<\infty$, $0<q<\infty$  or
$\cF=M^{s}_{p,q}$, where $0<s<1$ and $Q/(Q+s)<p,q<\infty$ or $0<s\le 1$, $Q/(Q+s)<p<\infty$ and $q=\infty$.
If $u\in\dot\cF(X)$, then $\cF$-quasievery point is a Lebesgue point of $u$.
Moreover, $u^*(x)=\lim_{r\to 0}u_{B(x,r)}$ is an $\cF$-quasicontinuous representative of $u$. 
\end{theorem}

Throughout the rest of this section, we assume that $\cF$ satisfies the assumptions of Theorem \ref{leb points}.
Since the condition $p>Q/(Q+s)$ guarantees only local integrability of Besov or Triebel--Lizorkin functions, 
we have to use restricted maximal functions
\[
\M_Ru(x)=\sup_{0<r<R} |u|_{B(x,r)}
\]
and
\[
\M_R^*u(x)=\sup_{2^k<R} |u|_{2^k}(x).
\]
Notice that if $u$ is locally integrable and $R<\infty$, then $\M_Ru$ and $\M_R^*u$ are almost everywhere finite. 
It is easy to see that there exists a constant $C=C(c_d)$ such that
\begin{equation}\label{comparability for restricted maximal functions}
C^{-1}\M_{C^{-1}R}u\le \M_R^*u\le C\M_{CR}u.
\end{equation}

Theorem \ref{HeTu} below follows from the proofs of Theorems 3.4, 4.7 and 4.8 in \cite{HeTu}.

\begin{theorem}\label{HeTu}
Let $0<R<\infty$. Then there exists a constant $C=C(c_d,s,p,q)$ such that
\[
\|\M_R^*u\|_{\dot \cF(X)}\le C\|u\|_{\dot\cF(X)}
\]
for all $u\in \dot\cF(X)$. 
\end{theorem}

We need the following simple lemma.

\begin{lemma}\label{uusi lemma} 
Suppose that $u=0$ outside $B=B(x_0,r)$ and that there exists $y_0\in X$ such that $d(x_0,y_0)=3r$.
Then there is a constant $C=C(c_d,s,p,q)$ such that $\|u\|_{\cF(X)}\le C(1+r^s)\|u\|_{\dot\cF(X)}$.
\end{lemma}

\begin{proof}
Let $(g_k)_{k\in\z}\in \D^s(u)$. Then, for almost every $x\in B$,
\[
|u(x)|\le 5^sr^s\Big(g(x)+\essinf_{y\in 4B\setminus 2B}g(y)\Big),
\]
where $g=\max\{g_k:r/2\le  2^{-k}<10r\}$. Since $B(y_0,r)\subset 4B\setminus 2B$ and $B\subset B(y_0,4r)$, 
it follows from the doubling property that $\mu(B)\le C\mu(4B\setminus 2B)$. Hence
\[
\begin{split}
\|u\|_{L^p(X)}
&=\|u\|_{L^p(B)}\le Cr^s\|g\|_{L^p(4B)}\\
&\le Cr^s\min\{\|(g_k)\|_{l^q(L^p(X))},\|(g_k)\|_{L^p(X, l^q)}\}.
\end{split}
\]
Since this holds for every $(g_k)_{k\in\z}\in \D^s(u)$, we have that
\[
\|u\|_{L^p(X)}\le Cr^s\|u\|_{\dot\cF(X)}
\]
and the claim follows.
\end{proof}


By combining Theorem \ref{HeTu} and Lemma \ref{uusi lemma}, we obtain a localized capacitary weak-type estimate for the restricted Hardy--Littlewood maximal function.

\begin{theorem}\label{cap weak type 2}
Let $B=B(x_0,r)$ be a ball and assume that the sphere $\{y: d(y,x_0)=30r\}$ is nonempty.
Then there exist constants $c=c(c_d)$ and $C=C(c_d,s,p,q,r)$ such that
\[
C_{\cF}\big(\{x\in B: \M_{r/c} u(x) >\lambda\}\big)
\le C\lambda^{-p}\|u\|_{\cF(X)}^p
\]
for all $u\in \cF(X)$.
\end{theorem}

\begin{proof} 
Let $c$ be the constant in \eqref{comparability for restricted maximal functions}. Then $\M_{r/c}u(x)\le c\M^*_ru(x)$.
Let $\varphi\ge0$ be a Lipschitz function such that $\varphi |_{5B}=1$ and $\varphi |_{X\setminus 6B}=0$.
Then, by \cite[Lemma 3.10 and Remark 3.11]{HIT}, $u\varphi\in \cF(X)$ 
and 
\begin{equation}\label{est}
\|u\ph\|_{\cF(X)}\le C\|u\|_{\cF(X)}.
\end{equation}
If $x\in B$, then $\M^*_ru(x)=\M^*_r(u\ph)(x)$.  Hence,
\[
\{x\in B: \M_{r/c} u(x) >\lambda\}\subset \{x\in X: c\lambda^{-1}\M^*_r (u\ph)(x) >1\}
\]
and so
\[
C_{\cF}\big(\{x\in B: \M_{r/c} u(x) >\lambda\}\big)
\le C\lambda^{-p}\|\M^*_r (u\ph)\|_{\cF(X)}^p.
\]
If $x\in X\setminus 10B$, then $\M^*_r(u\ph)(x)=0$. Thus,
by Lemma \ref{uusi lemma}, Theorem \ref{HeTu} and \eqref{est},
\[
\|\M^*_r(u\ph)\|_{\cF(X)}\le C\|\M^*_r(u\ph)\|_{\dot\cF(X)}\le C\|u\ph\|_{\dot\cF(X)}\le C\|u\|_{\cF(X)}
\]
and the claim follows.
\end{proof}

\begin{proof}[Proof of Theorem \ref{leb points}]
We may assume that $X$ contains at least two points. 
Then $X$ can be covered by balls $B_i=B(x_i,r_i)$, $i\in I$, where $I\subset\n$, such that the spheres $\{y: d(x_i,y)=30r_i\}$ are nonempty.




Theorem \ref{cap weak type 2} and a similar argument as in the proof of Theorem \ref{main thm} imply that, for every $i$,
there exists a set $E_i$ with $C_{\cF}(E_i)=0$ such that the limit $\lim_{r\to 0} u_{B(x,r)}=u^*(x)$ exists in $B_i\setminus E_i$. 
Moreover, for every $\eps>0$, there exists a set $U_{i}$ such that $C_{\cF}(U_i)<2^{-i}\eps$ and $u^*$ is continuous in 
$B_i\setminus U_i$.
Hence, the limit
\[
\lim_{r\to 0} u_{B(x,r)}=u^*(x)
\]
exists in $X\setminus \cup_{i\in I}E_i$ and $u^*$ is continuous in $X\setminus \cup_{i\in I}U_i$.
By \eqref{eq: r-subadd}, $C_{\cF}(\cup_{i\in I}E_i)=0$ and $C_{\cF}(\cup_{i\in I}U_i)<C\eps$.
\end{proof}

The following corollary of Theorem \ref{leb points} extends the results obtained in \cite{KL}, \cite{KiTu} and \cite{P} to the case 
$Q/(Q+s)<p<1$.

\begin{corollary}\label{leb points for Msp}
Let $0<s\le 1$, $Q/(Q+s)<p<\infty$ and $u\in \dot M^{s,p}(X)$. 
Then $M^{s,p}$-quasievery point is a Lebesgue point of $u$ and $u^*(x)=\lim_{r\to 0}u_{B(x,r)}$ is an $M^{s,p}$-quasicontinuous representative of $u$. 
\end{corollary}

\noindent {\bf Acknowledgements:} The research was supported by the Academy of Finland. 
Part of this research was conducted during the visit of the third author to Forschungs\-institut f\"ur Mathematik of ETH Z\"urich, and she wishes to thank the institute for the kind hospitality.

\vspace{0.5cm}
\noindent
\small{\textsc{T.H.},}
\small{\textsc{Department of Mathematics},}
\small{\textsc{P.O. Box 11100},}
\small{\textsc{FI-00076 Aalto University},}
\small{\textsc{Finland}}\\
\footnotesize{\texttt{toni.heikkinen@aalto.fi}}

\vspace{0.5cm}
\noindent
\small{\textsc{P.K.},}
\small{\textsc{Department of Mathematics and Statistics},}
\small{\textsc{P.O. Box 35},}
\small{\textsc{FI-40014 University of Jyv\"askyl\"a},}
\small{\textsc{Finland}}\\
\footnotesize{\texttt{pekka.j.koskela@jyu.fi}}

\vspace{0.3cm}
\noindent
\small{\textsc{H.T.},}
\small{\textsc{Department of Mathematics and Statistics},}
\small{\textsc{P.O. Box 35},}
\small{\textsc{FI-40014 University of Jyv\"askyl\"a},}
\small{\textsc{Finland}}\\
\footnotesize{\texttt{heli.m.tuominen@jyu.fi}}

\end{document}